\documentclass[12pt, a4paper, leqno]{amsart}
\usepackage[utf8]{inputenc}
\usepackage{amsmath}
\usepackage{amsfonts}
\usepackage{amssymb}
\usepackage{times}
\usepackage[T1]{fontenc} 
\usepackage{url} 
\usepackage{color,esint}
\usepackage[colorlinks]{hyperref} 
\usepackage{mdframed}
\usepackage{pgf,tikz,pgfplots}
\pgfplotsset{compat=1.14}
\usepackage{mathrsfs}
\usetikzlibrary{arrows}

\let\oldmarginpar\marginpar
\renewcommand\marginpar[1]{\-\oldmarginpar[\raggedleft\footnotesize #1]%
{\raggedright\footnotesize #1}}

\usepackage{amsthm}
\theoremstyle{plain}
\newtheorem{thm}[equation]{Theorem}
\newtheorem{lem}[equation]{Lemma}

\newtheorem{cor}[equation]{Corollary}

\theoremstyle{definition}
\newtheorem{defn}[equation]{Definition}

\theoremstyle{remark}

\numberwithin{equation}{section}

\newcommand{\R}{\mathbb{R}}

\newcommand{\Rn}{\mathbb{R}^n}

\def\essinf{\operatornamewithlimits{ess\,inf}}
\def\essliminf{\operatornamewithlimits{ess\,\lim\,inf}}
\def\supp{\operatornamewithlimits{supp}}

\def\le{\leqslant}
\def\leq{\leqslant}
\def\ge{\geqslant}
\def\geq{\geqslant}
\def\rho{\varrho}
\def\vartheta{\theta}
\renewcommand{\phi}{\varphi}
\renewcommand{\epsilon}{\varepsilon}

\def\lip{\operatorname{Lip}}
\def\supp{\operatorname{supp}}

\def\diam{\qopname\relax o{diam}}

\def\spt{\qopname\relax o{spt}}

\def\phix{{\phi(\cdot)}}

\def\loc{{\rm loc}}
\def\ve{\varepsilon}

\newcommand{\ainc}[1]{{{\normalfont(aInc){\ensuremath{_{#1}}}}}}
\newcommand{\adec}[1]{{{\normalfont(aDec){\ensuremath{_{#1}}}}}}

\newcommand{\azero}{{{\normalfont(A0)}}}
\newcommand{\aone}{{{\normalfont(A1)}}}
\newcommand{\aonen}{{{\normalfont(A1-$n$)}}}

\date{\today}

\usepackage{url}

\usepackage{color}
\definecolor{blau}{rgb}{0.1,0.0,0.9}

\newcounter{komcounter}
\numberwithin{komcounter}{section}

\begin{document}

\title{The Kellogg property under generalized growth conditions}
\author{Petteri Harjulehto}
\address{P.\ Harjulehto: 
Department of Mathematics and Statistics,
FI-20014 University of Turku, Finland}
\email{petteri.harjulehto@utu.fi}

\author{Jonne Juusti}
\address{J.\ Juusti: 
Department of Mathematics and Statistics,
FI-20014 University of Turku, Finland}
\email{jthjuu@utu.fi}

\begin{abstract}
We study minimizers  of the Dirichlet $\phi$-energy integral with generalized Orlicz growth. We prove the Kellogg property, the set of irregular points has zero capacity, and give characterizations of semiregular boundary points. The results are new ever for the special cases double phase and Orlicz growth.
\end{abstract}

\keywords{Dirichlet energy integral, irregular boundary point, semiregular boundary point,  minimizer, superminimizer,  generalized Orlicz space, Musielak--Orlicz spaces, nonstandard growth, variable exponent, double phase}

\subjclass[2010]{49N60, 35J67, 35J20}

\maketitle

\section{Introduction}

We study minimizers of the Dirichlet energy integral in a bounded domain $\Omega \subset \Rn$ with  boundary values:
\[
\inf \int_\Omega \phi(x,|\nabla u|) \, dx,
\]
where the integral is taken over all $u \in W^{1,\phix}(\Omega)$ with $u - f \in W_0^{1,\phix}(\Omega)$.
We assume that strictly convex $\phi$ has the generalized Orlicz growth and satisfies \azero{}, \aone{}, \aonen{}, \ainc{} and \adec{}. These conditions for the generalized Orlicz function are widely used, see for example \cite{HarHT17, Has15, HasO19, Kar18, YanYY19}.
Our results include as special cases  the constant exponent case $\phi(x,t) = t^p$, the Orlicz case $\phi(x,t) = \phi(t)$, the variable exponent case $\phi(x,t) = t^{p(x)}$, and the double phase case $\phi(x,t) = t^p + a(x) t^q$.
Boundary regularity has been recently studied in the variable exponent case for example in \cite{AdaBB14, AdaL16, LatLT11, Luk10, RagT16, TacU17}, in Orlicz case for example in \cite{Gro02, KriM10, Ok18}, in double phase case in \cite{BarKLU18}, and in the generalized Orlicz case in \cite{HarH19, Kar_pp}.
Note the survey \cite{Chl18}, that includes more references of variational problems and partial differential equations of this type.
We also mention the books \cite{CruF13}, that presents the fundamentals of variable Lebesgue spaces and how they relate to harmonic analysis, and \cite{YanLK17}, that surveys the theory of Musielak-Orlicz Hardy spaces.

Let $f \in C(\partial \Omega)$ be a boundary value function and $H_f$ the corresponding (continuous) 
minimizer, see Section~\ref{sec:regular} for definitions. A boundary point $x \in \partial \Omega$ is regular if  $\lim_{y \to x, y \in \Omega} H_f(y) = f(x)$ for all $f$.
Otherwise  the boundary point is irregular. An irregular boundary point is semiregular if the limit exists for all $f$. Precise definitions can be found from Definitions~\ref{defn:regular} and \ref{defn:semiregular}.
Our main goal is to prove the Kellogg property: the set of irregular boundary points has zero capacity.
This is our Theorem \ref{thm:main}. In the variable exponent case this was first proved in \cite{LatLT11}, and later with a different proof in \cite{AdaBB14}. Then we prove characterizations of semiregular boundary points, Theorem~\ref{thm:semireg}, showing for example that the boundary point $x_0$ is semiregular if and only if  it has a neighbourhood $V$ such that capacity of $V \cap \partial \Omega$ is zero. In the variable exponent case these have been proved in 
\cite{AdaBB14}. In this paper we use ideas from \cite{AdaBB14}.
To best of our knowledge, our results are new in even the special cases of the double phase growth and the Orlicz growth.


\section{Preliminaries}
\label{sect:elementary} 

Throughout this paper, we assume that $\Omega \subset \Rn$ is a bounded domain, i.e. an open and connected set.
The following definitions are as in \cite{HarH19b}, which we use as a general reference to background theory in generalized Orlicz spaces.

\begin{defn}
We say that $\phi: \Omega\times [0, \infty) \to [0, \infty]$ is a 
\textit{weak $\Phi$-function}, and write $\phi \in \Phi_w(\Omega)$, if 
the following conditions hold
\begin{itemize}
\item For every $t \in [0, \infty)$ the function $x \mapsto \phi(x, t)$ is measurable and for every $x \in \Omega$ the function $t \mapsto \phi(x, t)$ is non-decreasing.
\item $\displaystyle \phi(x, 0) = \lim_{t \to 0^+} \phi(x,t) =0$  and $\displaystyle \lim_{t \to \infty}\phi(x,t)=\infty$ for every $x\in \Omega$.
\item The function $t \mapsto \frac{\phi(x, t)}t$ is 
$L$-almost increasing for $t>0$ uniformly in $\Omega$. "Uniformly" means that $L$ 
is independent of $x$.
\end{itemize}
If $\phi\in\Phi_w(\Omega)$ is additionally convex and left-continuous, then $\phi$ is a 
\textit{convex $\Phi$-function}, and we write $\phi \in \Phi_c(\Omega)$.
\end{defn}

Two functions $\phi$ and $\psi$ are \textit{equivalent}, 
$\phi\simeq\psi$, if there exists $L\ge 1$ such that 
$\psi(x,\frac tL)\le \phi(x, t)\le \psi(x, Lt)$ for every $x \in \Omega$ and every $t>0$.
Equivalent $\Phi$-functions give rise to the same space with 
comparable norms.

\subsection*{Assumptions}\label{sect:assumptions}

Let us write $\phi^+_B (t) := \sup_{x \in B} \phi(x, t)$ and $\phi^-_B (t) := \inf_{x \in B} \phi(x, t)$; and abbreviate $\phi^\pm := \phi^\pm_\Omega$.
We state some assumptions for later reference. 

\begin{itemize}
\item[(A0)]
There exists $\beta \in(0,1)$ such that $\phi(x, \beta) \le 1 \le \phi(x,1/\beta)$ for almost every $x$.
\item[(A1)]\label{defn:a1}
There exists $\beta\in (0,1)$ such that,
for every ball $B$ and a.e.\ $x,y\in B \cap \Omega$,
\[
\beta \phi^{-1}(x, t) \le \phi^{-1} (y, t) 
\quad\text{when}\quad 
t \in \bigg[1, \frac{1}{|B|}\bigg].
\]
\item[(A1-$n$)]
There exists $\beta\in (0,1)$ such that, for every ball $B\subset \Omega$, 
\[
\phi^+_B (\beta t) \le  \phi^-_B (t)
\quad\text{when}\quad
t \in \big[1, \tfrac1{\diam (B)}\big].
\]
\item[(A2)]
For every $s > 0$ there exist $\beta \in (0, 1]$ and $h \in L^1(\Omega) \cap L^\infty (\Omega)$ such that
\[
\beta \phi^{-1}(x,t) \leq \phi^{-1}(y,t)
\]
for almost every $x, y \in \Omega$ and every $t \in [h(x) + h(y), s]$.
\item[(aInc)] 
There exist $p>1$ and $L\ge 1$ such that $t \mapsto \frac{\phi(x,t)}{t^{p}} $ is $L$-almost increasing in $(0,\infty)$.
\item[(aDec)] 
There exist $q>1$ and $L\ge 1$ such that $t \mapsto \frac{\phi(x,t)}{t^{q}} $ is $L$-almost decreasing in $(0,\infty)$.
\end{itemize}
We write (Inc) 
if the ratio is increasing rather than just almost increasing, similarly for (Dec).

We say that $\phi$ is \textit{doubling} if there exists a constant $L\ge 1$ such that $\phi(x, 2t) \le L \phi(x, t)$ for every $x \in \Omega$ and every $t>0$.
By Lemma~2.2.6 of \cite{HarH19b} doubling is equivalent to (aDec).
If $\phi$ is doubling with constant $L$, then 
by iteration
\begin{equation}\label{equ:doubling_iteration}
\phi(x,t) \le L^2 \Big( \frac{t}{s}\Big)^{Q} \phi(x,s) 
\end{equation}
for every $x \in \Omega$ and every $0<s<t$, where $Q= \log_2(L)$, e.g.\ \cite[Lemma~3.3]{BjoB11}. If $\phi$ is doubling, then \eqref{equ:doubling_iteration} yields that $\simeq$ implies $\approx$.
On the other hand, $\approx$ always implies $\simeq$ since the function $t \mapsto \frac{\phi(x, t)}t$ is 
almost increasing;
hence $\simeq$ and $\approx$ are equivalent in the doubling case. Note that doubling also yields that $\phi(x, t+s)\lesssim \phi(x, t) + \phi(x, s)$.

Assumptions (A0) and (aDec) imply that 
$\phi(x,1) \lesssim \beta^{-q} \phi(x,\beta) \le  \beta^{-q}$
 and $\phi(x, 1) \gtrsim \beta^q \phi(x, 1/\beta) \ge \beta^q$, and thus $\phi(x,1)\approx 1$. 
If $\phi \in \Phi_c(\Omega)$ satisfies (aDec), then $\phi(x,t)$ is finite for every $x \in \Omega$ and $t \geq 0$, and convexity implies that $\phi$ is continuous.
The conditions (A1) and (A1-$n$) can be used also in cubes instead of balls, see 
Lemmas~2.10 and 2.11 in \cite{HarHT17}.

The next table contains a interpretation of
the assumptions for four $\Phi$-functions.  The table is a combination of tables from \cite[Table 7.1]{HarH19b} and
\cite{HarHT17}.

\medskip
\centerline{
\begin{tabular}{l|ccccc}
$\phi(x,t)$ & (A0)& (A1)& (A1-$n$)& (aInc) & (aDec) \\
\hline
$t^{p(x)}a(x)$ & $a\approx 1$ & $\frac1p\in C^{\log}$ & $\frac1p\in C^{\log}$ & $p^->1$ & $p^+<\infty$ \\
$t^{p(x)}\log(e+t)$ & $\text{true}$ & $\frac1p\in C^{\log}$ & $\frac1p\in C^{\log}$ & $p^->1$ & $p^+<\infty$ \\
$t^p + a(x) t^q$ & $a\in L^\infty$ & $a\in C^{0, \frac np (q-p)}$ & $a\in C^{0, q-p}$ & $p>1$ & $q<\infty$ \\
$\phi(t)$ & \text{true} & \text{true} & \text{true} & \text{same} & \text{same} \\
\end{tabular}
}

\subsection*{Generalized Orlicz spaces}\label{sect:genOrlicz}

We recall some definitions. 
We denote by $L^0(\Omega)$ the set of measurable 
functions in $\Omega$.  

\begin{defn}\label{def:Lphi}
Let $\phi \in \Phi_w(\Omega)$ and define the \textit{modular} 
$\varrho_\phix$ for $f\in L^0(\Omega)$ by 
  \begin{align*}
    \varrho_\phix(f) &:= \int_\Omega \phi(x, |f(x)|)\,dx.
  \end{align*}
The \emph{generalized Orlicz space},
also called Musielak--Orlicz space,
is defined as the set 
  \begin{align*}
    L^\phix(\Omega) &:= \big\{f \in L^0(\Omega) \colon
      \lim_{\lambda \to 0^+} \varrho_\phix(\lambda f) = 0\big\}
  \end{align*}
equipped with the (Luxemburg) norm 
  \begin{align*}
    \|f\|_{L^\phix(\Omega)} &:= \inf \Big\{ \lambda>0 \colon \varrho_\phix\Big(
      \frac{f}{\lambda}  \Big) \leq 1\Big\}.
  \end{align*}
If the set is clear from the context we abbreviate $\|f\|_{L^\phix(\Omega)}$ by $\|f\|_{\phix}$.
\end{defn}

H\"older's inequality holds in generalized Orlicz spaces with a constant $2$, without restrictions on the $\Phi_w$-function
\cite[Lemma~3.2.13]{HarH19b}:
\[
\int_\Omega |f|\, |g|\, dx \le 2 \|f\|_\phix \|g\|_{\phi^*(\cdot)}.
\]

\begin{defn}
A function $u \in L^\phix(\Omega)$ belongs to the
\textit{Orlicz--Sobolev space $W^{1, \phix}(\Omega)$} if its weak partial derivatives $\partial_1 u, \ldots, \partial_n u$ exist and belong to the space $L^{\phix}(\Omega)$.
For $u \in W^{1,\phix}(\Omega)$, we define the norm
\[
\| u \|_{W^{1,\phix}(\Omega)} := \| u \|_{\phix} + \| \nabla u \|_{\phix}.
\]
Here $\| \nabla u \|_{\phix}$ is a shortening of $\big{\|} | \nabla u | \big{\|}_{\phix}$.
By \cite[Lemma~6.1.5]{HarH19b} the definition above is valid.
Again, if $\Omega$ is clear from the context, we abbreviate $\| u \|_{W^{1,\phix}(\Omega)}$ by $\| u \|_{1,\phix}$.
\end{defn}

To study boundary value problems, we need a concept of weak boundary 
value spaces. 

\begin{defn}
$W^{1,\phi(\cdot)}_0 (\Omega)$ is the closure of $C^\infty_0(\Omega)$ in 
$W^{1,\phi(\cdot)} (\Omega)$. 
\end{defn}

\subsection*{Capacity and fine properties of functions}\label{sect:capacity}
Fine properties of Sobolev functions can be studied by different capacities. 
Here we use the generalized Orlicz $\phix$-capacity defined as follows. 
 
\begin{defn} 
Let $E \subset \Rn$. Then the \emph{generalized Orlicz $\phix$-capacity of $E$} is defined 
by 
\[
C_\phix (E) := \inf_{u\in S_\phix(E)} \int_{\Rn} \phi(x,|u|) + \phi(x, |\nabla u|) \, dx,
\]
where the infimum is taken over the set $S_\phix(E)$ of all functions 
$u \in W^{1, \phix} (\Rn)$ with $u \ge 1$ in an open set containing $E$.
\end{defn}

If $\phi \in \Phi_c(\Rn)$ satisfies (aDec) and (aInc), then capacity has the following properties, see \cite[Section~3]{BarHH18}.

 \begin{enumerate}
  \item[(C1)] $C_{\phix}(\emptyset)=0$.
  \item[(C2)] If $E_1\subset E_2 \subset \Rn$, then $C_{\phix}(E_1)\leqslant
    C_{\phix}(E_2)$.
  \item[(C3)] If $E \subset \Rn$, then
    \begin{equation*}
      C_{\phix}(E)=\inf_{\substack{E\subset U\\ 
          U\ \mathrm{open}}}C_{\phix}(U).
    \end{equation*}
  \item[(C4)] If $E_1, E_2 \subset \Rn$, then
    \begin{equation*}
      C_{\phix}(E_1\cup E_2)+C_{\phix}(E_1\cap E_2)\leqslant
      C_{\phix}(E_1)+C_{\phix}(E_2).
    \end{equation*}
  \item[(C5)] If $K_1\supset K_2\supset\cdots$ are compact sets, then
    \begin{equation*}
      \lim_{i\to\infty}C_{\phix}(K_i)
      =C_{\phix}\big(\cap_{i=1}^{\infty} K_i\big).
    \end{equation*}
\item[(C6)] 
For $E_1\subset E_2\subset\cdots\subset\Rn$, 
  \begin{equation*}
    \lim_{i\to\infty}C_{\phix}(E_i)=C_{\phix}\big(\cup_{i=1}^{\infty}E_i\big).
  \end{equation*}
\item[(C7)] 
For $E_i\subset \Rn$, 
  \begin{equation*}
    C_{\phix}\big(\cup_{i=1}^{\infty}E_i\big)\leq\sum_{i=1}^{\infty}C_{\phix}(E_i).
  \end{equation*}
 \end{enumerate}

A  function $f: \Omega \rightarrow[-\infty, \infty]$ is 
\textit{$\phix$-quasicontinuous} if for every 
$\epsilon>0$ there exists an open set $U$ such that $C_{\phix}(U)<\epsilon$ and 
$f|_{\Omega\setminus U}$ is continuous.
We say that a claim holds \emph{$\phix$-quasieverywhere} if it holds everywhere 
except in a set of $\phix$-capacity zero.

Suppose that $u$ can be approximated by continuous functions in  $W^{1,\phix}(\Omega)$.
Then a standard argument (e.g. \cite[Theorem~11.1.3]{DieHHR11}) shows that every $u \in W^{1,\phix}(\Omega)$ has a representative, which is quasicontinuous in $\Omega$, provided that $\phi \in \Phi_c(\Omega)$ satisfies (aInc) and (aDec).
By \cite[Theorem~6.4.7]{HarH19b}, smooth functions are dense in $W^{1,\phix}(\Omega)$, if $\phi$ satisfies (A0), (A1), (A2) and (aDec).
By \cite[Lemma~4.2.3]{HarH19b}, (A2) is not needed, if $\Omega$ is bounded.
Hence we get the following lemma.

\begin{lem}\label{lem:density}
Let $\phi \in \Phi_c(\Omega)$ satisfy (A0), (A1) and (aDec).
Then for every $u \in W^{1, \phix}(\Omega)$, there exists a sequence of function from $C^\infty(\Omega) \cap W^{1, \phix}(\Omega)$ converging to $u$ in $W^{1, \phix}(\Omega)$.
\end{lem}

If $u \in W^{1, \phix}_0(D)$ and  $D \subset \Omega$, then the zero extension of 
$u$ belongs to $W^{1, \phix}(\Omega)$ since $u$ can be approximated 
by $C^\infty_0(D)$-functions. The next lemma concerns the opposite implication. 

\begin{lem}[Theorem~2.10 in \cite{HarH19}]\label{lem:nolla_ulkopuolella}
Let $\phi \in \Phi_c(\Omega)$ satisfy (A0), (A1), (aInc) and (aDec) and let $D \Subset \Omega$ be open.
If $u \in W^{1, \phix}(\Omega)$ and $u=0$ in $\Omega\setminus D$, 
then $u \in W^{1, \phix}_0(D)$.
Moreover, if $u$ is non-negative, then there exist non-negative $u_i\in W^{1, \phix}_0(D)$ 
with $\spt u_i \Subset D$, $\{u_i \neq 0\} \subset \{u \neq 0\}$ and $u_i\to u$ in $W^{1, \phix}(D)$.
\end{lem}


\section{Regular boundary points}\label{sec:regular}

\begin{defn}
Let $\phi \in \Phi_w(\Omega)$ and $f \in W^{1, \phix}(\Omega)$. 
We say that $u\in W^{1, \phix}(\Omega)$ is a \textit{minimizer} with boundary values $f\in W^{1, \phix}(\Omega)$ if $u-f \in W^{1, \phix}_0(\Omega)$ and
\begin{equation}\label{equ:min}
\int_\Omega \phi(x, |\nabla u|) \, dx \le \int_\Omega \phi(x, |\nabla (u + v)|) \, dx
\end{equation}
for all $v\in W^{1, \phix}_0(\Omega)$. 

If the inequality is assumed only for all non-negative or non-positive $v$, then
$u$ is called a \emph{superminimizer} or  \emph{subminimizer}, respectively.
\end{defn}

In the next lemma, we show that in some cases the set $W_0^{1,\phix}(\Omega)$ in the definition above can be replaced with $C_0^\infty(\Omega)$.

\begin{lem}\label{lem:min_cont}
Let $\phi \in \Phi_c$ satisfy (A0), (A1) (aInc) and (aDec).
Then $u$ is a minimizer, if and only if inequality \eqref{equ:min} holds for all $v \in C_0^\infty(\Omega)$.
The function $u$ is a superminimizer (subminimizer), if and only if \eqref{equ:min} holds for all positive (negative) $v \in C_0^\infty(\Omega)$.
\end{lem}

\begin{proof}
If $u$ is a minimizer, then it is trivial that \eqref{equ:min} holds for all $v \in C_0^\infty(\Omega)$.

Suppose then that \eqref{equ:min} holds for all $v \in C_0^\infty(\Omega)$.
Let $w \in W_0^{1,\phix}(\Omega)$, and let $w_i \in C_0^\infty$ be a sequence of functions converging to $w$ in $W^{1,\phix}(\Omega)$.
Denote $w' := u+w$ and $w_i' := u+w_i$.
By \cite[Lemma~3.2.11]{HarH19b}
\[
\int_\Omega \phi(x,|\nabla w'|-|\nabla w_i'|) \,dx
\leq \max \{\big{\|} |\nabla w'|-|\nabla w_i'| \big{\|}_\phix, \big{\|} |\nabla w'|-|\nabla w_i'| \big{\|}_{\phix}^q \},
\]
where $q$ is the exponent from (aDec). Since $\big{|}|\nabla w'|-|\nabla w_i'|\big{|} \leq |\nabla (w'-w_i') |$, we get
\[
\big{\|} |\nabla w'|-|\nabla w_i'| \big{\|}_\phix \leq \|\nabla (w'-w_i') \|_\phix = \|\nabla (w-w_i) \|_\phix.
\]
Since $\|\nabla (w-w_i) \|_\phix \to 0$ as $i \to \infty$, it follows that
\[
\int_\Omega \phi(x,|\nabla w'|-|\nabla w_i'|) \,dx \to 0 \quad \text{as} \quad i \to \infty.
\]
By \cite[Lemma~2.2.6]{HarH19b} $\phi$ satisfies (Dec) (not only \adec{}). Since $\lim_{\lambda \to 0^+} \varrho_\phix(\lambda\nabla w') = 0$, (aDec) implies that $\varrho_\phix(\nabla w')$ is bounded. It now follows from \cite[Lemma~3.1.6]{HarH19b} that $|\varrho_\phix(\nabla w') - \varrho_\phix(\nabla w_i')|$ approaches zero as $i \to \infty$, and we therefore have
\[
\lim_{i \to \infty} \int_\Omega \phi(x,|\nabla w_i'|) \,dx = \int_\Omega \phi(x,|\nabla w'|) \,dx = \int_\Omega \phi(x,|\nabla (u+w)|) \,dx.
\]
By our assumption, for every $i$ we have
\[
\int_\Omega \phi(x,|\nabla u|) \,dx \leq \int_\Omega \phi(x,|\nabla w_i'|) \,dx.
\]
Combining the above estimate and limit gives
\[
\int_\Omega \phi(x,|\nabla u|) \,dx \leq \int_\Omega \phi(x,|\nabla w'|) \,dx = \int_\Omega \phi(x,|\nabla (u+w)|) \,dx,
\]
which shows that $u$ is a minimizer.

The claim regarding superminimizers is proved similarly.
The only difference is that every function in the sequence $\{w_i\}$ must be non-negative.
Suppose that $w \in W_0^{1,\phix}(\Omega)$ is non-negative.
By definition, there is a sequence of functions $w_i \in C_0^\infty(\Omega)$ converging to $w$.
But from the definition alone we can't deduce that the functions $w_i$ are non-negative.
Instead, we use Lemma~\ref{lem:nolla_ulkopuolella}:
let $\tilde{w}_i \in W_0^{1,\phix}(\Omega)$ be a sequence of non-negative functions such that $\spt \tilde{w}_i \Subset \Omega$ and $\tilde{w}_i \to w$ in $W^{1,\phix}(\Omega)$.
The proof of \cite[Theorem~6.4.7]{HarH19b} shows that for every $i$, there is a sequence of functions $\eta_{ij} \in C^\infty(\Omega) \cap W^{1,\phix}(\Omega)$ converging to $\tilde{w}_i$.
Moreover, since the functions $\eta_{ij}$ are obtained using standard mollifiers on $\tilde{w}_i$, and $\spt \tilde{w}_i \Subset \Omega$, it follows that $\eta_{ij} \in C_0^\infty(\Omega)$ and every $\eta_{ij}$ is non-negative.
For every $i$, we choose an index $j_i$ with
\[
\|\tilde{w}_i-\eta_{ij_i}\|_{1,\phix} < i^{-1}. 
\]
Then $\eta_{ij_i} \to w$ in $W^{1,\phix}(\Omega)$.
This completes the proof in the case of superminimizers.
The claim for subminimizers follows from the fact that $-u$ is a superminimizer.
\end{proof}

We denote by $H(f)$  the minimizer with boundary values $f \in W^{1, \phix} (\Omega)$. If $f: \partial \Omega \to \R$ is Lipschitz on the boundary of $\Omega$, 
then it can be, by McShane extension, extended to $\Rn$ as a bounded Lipschitz function. 
The extension of $f$ can be used in the above definition as weak boundary value, 
$u-f\in W^{1,\phix}_0(\Omega)$. 
For $g \in C(\partial \Omega)$ we define
\[
H_g(x) := \sup_{\lip(\partial\Omega) \ni f\le g} H(f)(x).
\]
This definition is based on the fact that continuous functions can be approximated by Lipschitz functions. 

The following theorem gives sufficient conditions for existence, uniqueness and continuity of minimizer with bounded boundary values.

\begin{thm}[Theorem~6.2 in \cite{HarH19}]\label{thm:H(f)}
Let $\phi \in \Phi_c(\Omega)$ satisfy (aInc) and (aDec).
Then for every function $f \in W^{1,\phix}(\Omega) \cap L^\infty(\Omega)$, there exists a minimizer $H(f)$. 

If $\phi$ is strictly convex and satisfies (A0), the minimizer is unique, and if (A1-$n$) holds, then it is continuous. 
\end{thm}

\begin{defn}\label{defn:regular}
Let $\Omega \subset \Rn$. 
We say that $x \in \partial \Omega$ is \emph{regular} if
\[
\lim_{y \to x, y \in \Omega} H_f(y) = f(x)
\]
for all $f \in C(\partial \Omega)$. A boundary point is \emph{irregular} if it is not regular.
\end{defn}

This means that the minimizer attains the boundary 
values not only in a Sobolev sense but point-wise.

We finish this section with the definition of quasiminimizers.

\begin{defn}
A function $u \in W_{\loc}^{1,\phix}$ is a \textit{local quasiminimizer} of the $\phix$-energy if there is a constant $K \geq 1$ such that
\[
\int_{\{v \neq 0\}} \phi(x,|\nabla u|) \,dx \leq K \int_{\{v \neq 0\}} \phi(x,|\nabla (u+v)|) \,dx
\]
for all $v \in W^{1,\phix}(\Omega)$ with $\spt v \subset \Omega$.

If the inequality is assumed only for all non-negative or non-positive $v$, then
$u$ is called a \emph{local quasisuperminimizer} or  \emph{local quasisubminimizer}, respectively.
\end{defn}

By \cite[Lemma~3.4]{HarHT17}, if $\phi$ satisfies (A0), (A1) and (aDec), then every $v \in W^{1,\phix}(\Omega)$ with $\spt v \subset \Omega$ belongs to $W_0^{1,\phix}(\Omega)$.
It then follows that every minimizer is also a local quasiminimizer.

\section{Quasisuperminimizer equals lsc-reguralization quasieverywhere}

\begin{lem}[Theorem 4.4 in \cite{HarH19}]\label{lem:lsc}
Let $\phi \in \Phi_c(\Omega)$ satisfy (A0), (A1-$n$), (aInc) and (aDec).
Let $u$ be a local quasisuperminimizer which is bounded from below and set
\[
u^*(x) := \essliminf_{y \to x} u(y).
\]
Then $u^*$ is lower semicontinuous and  $u= u^*$ almost everywhere. 

If $u$ is additionally locally bounded, then every point is a Lebesgue point of $u^*$.
\end{lem}

In the lemma above, the  function $u^*$ is called the \textit{lsc-regularization} of $u$.
We say that $u$ is \textit{lsc-regularized}, if $u = u^*$.
In this section we prove that if $u$ is a quasicontinuous quasisuperminimizer, then $u = u^*$ quasieverywhere.
To accomplish this, we need the following lemma and its corollary.

\begin{lem}\label{lem:quasicont_represent}
Let $\phi \in \Phi_c(\Rn)$ satisfy (A0), (A1), (aInc) and (aDec).
Let $B$ be a ball and suppose that $u \in W^{1, \phix}(\Rn)$ is such that $\spt u \Subset B$.
Then there exists a set $E \subset B$ of zero capacity, such that
\[
\hat{u}(x) := \lim_{r \to 0} \fint_{B(x,r)} u(y) \, dy
\]
exists for every $x \in B \setminus E$. The function $\hat{u}$ is the quasicontinuous representative of $u$.
\end{lem}

\begin{proof}
To prove this claim, we follow the proofs of Propositions 4.4 and 4.5 and Theorem 4.6 of \cite{HarH04}, where a similar claim was proven for the variable exponent case.

First we show that $Mu \in W^{1, \phix}(3B)$ and
\begin{equation}\label{equ:maximal}
\| Mu \|_{W^{1, \phix}(3B)} \leq c\| u \|_{W^{1, \phix}(3B)}.
\end{equation}
Here $Mu$ denotes the Hardy--Littlewood maximal function.
Since $u, |\nabla u| \in L^{\phix}(3B)$, it follows by \cite[Lemma~4.2.3]{HarH19b} and \cite[Theorem~4.3.6]{HarH19b} that $Mu, M|\nabla u| \in L^{\phix}(3B)$, and further
\[
\| Mu \|_{L^{\phix}(3B)}
\leq c\| u \|_{L^{\phix}(3B)}
\quad\text{and}\quad
\| M|\nabla u| \|_{L^{\phix}(3B)}
\leq c\| \nabla u \|_{L^{\phix}(3B)}.
\]
Note that \cite[Lemma~4.2.3]{HarH19b} is needed here, because \cite[Theorem~4.3.6]{HarH19b} requires the assumption (A2).
By \cite[Lemma~6.1.6]{HarH19b}, we have $u \in W^{1, p}(B)$, where $p>1$ is such that $\phi$ satisfies \ainc{p}.
Since $\spt u \Subset B$, it follows that $u$ in $W^{1, p}(\Rn)$.
From \cite{Kin97} it follows that $|\nabla Mu| \leq M|\nabla u|$ almost everywhere in $\Rn$. Hence
\[
\| \nabla Mu \|_{L^{\phix}(3B)}
\leq \| M|\nabla u| \|_{L^{\phix}(3B)}
\leq c\| \nabla u \|_{L^{\phix}(3B)},
\]
and \eqref{equ:maximal} now follows.

Then we show that for $\lambda > 0$ we have
\begin{equation}\label{equ:cap}
C_{\phix}(\{Mu > \lambda\} \cap B) \leq
c \max \{\| u / \lambda \|_{W^{1, \phix}(3B)}, \| u / \lambda \|_{W^{1, \phix}(3B)}^q\},
\end{equation}
where $q$ is such that $\phi$ satisfies \adec{q}.
Because $Mu$ is lower semi-continuous, the set $\{Mu > \lambda\}$ and its intersection with $B$ are open.
Let $\eta \in C_0^\infty(3B)$ be such that $\eta = 1$ in $2B$ and $0 \leq \eta \leq 1$ in $3B$.
Then we may use $\eta Mu / \lambda \in W^{1, \phix}(3B)$ as a test function for capacity of $\{Mu > \lambda\} \cap B$. Since $\eta Mu / \lambda = 0$ outside $3B$, we get
\begin{align*}
C_{\phix}(\{Mu > \lambda\} \cap B) &
\leq \int_{\Rn} \phi\left( x, \left| \eta M\frac{u}{\lambda}\right|\right) + \phi \left(x, \left| \nabla \left( \eta M\frac{u}{\lambda} \right)\right|\right) \, dx \\
& = \int_{3B} \phi\left( x, \left| \eta M\frac{u}{\lambda}\right|\right) + \phi \left(x, \left| \nabla \left( \eta M\frac{u}{\lambda} \right)\right|\right) \, dx \\
& \leq c \max \left\{ \left\| \eta M\frac{u}{\lambda} \right\|_{W^{1, \phix}(3B)}, \left\| \eta M\frac{u}{\lambda} \right\|_{W^{1, \phix}(3B)}^q\right\},
\end{align*}
where the last inequality follows by \cite[Lemma~3.2.11]{HarH19b}.
Now
\begin{align*}
\left\| \eta M\frac{u}{\lambda} \right\|_{W^{1, \phix}(3B)} &
= \left\| \eta M\frac{u}{\lambda} \right\|_{L^{\phix}(3B)} + \left\| \eta \nabla M\frac{u}{\lambda} + (\nabla \eta)M\frac{u}{\lambda} \right\|_{L^{\phix}(3B)} \\
& \leq \left\| \eta M\frac{u}{\lambda} \right\|_{L^{\phix}(3B)} + \left\| \eta \nabla M\frac{u}{\lambda} \right\|_{L^{\phix}(3B)} + \left\| (\nabla \eta) M\frac{u}{\lambda} \right\|_{L^{\phix}(3B)} \\
& \leq (1 + \| \nabla \eta \|_\infty) \left\| M\frac{u}{\lambda} \right\|_{W^{1, \phix}(3B)}.
\end{align*}
The first inequality follows from triangle inequality (\cite[Lemma~3.2.2]{HarH19b}), and the second from the fact that $\eta \leq 1$.
Since $\| \nabla \eta \|_\infty$ does not depend on $u$, it can be treated as constant depending only on $|B|$.
Inequality \eqref{equ:cap} then follows from \eqref{equ:maximal}.

Next we construct the set $E$.
By \cite[Lemma~3.4]{HarHT17}, $u\in W_0^{1,\phix}(B)\subset W_0^{1,\phix}(\Rn)$.
Let $\{u_i\}$ be a sequence of continuous functions converging to $u$ in $W^{1,\phix}(\Rn)$ such that $\|u-u_i\|_\phix \leq 2^{-2i}$.
Define the sets
\[
U_i := \{ M(u-u_i) > 2^{-i} \} \cap B, \quad
V_i := \bigcup_{j=i}^\infty U_j \quad\text{and}\quad
E := \bigcap_{j=1}^\infty V_j.
\]
By \eqref{equ:cap} we have $C_\phix(U_i) \leq c2^{-i}$, and therefore $C_\phix(V_i) \leq c2^{1-i}$ by subadditivity.
Since $E$ is contained in every $V_i$, it follows that $C_\phix(E) = 0$.

To complete the proof, we show that $\hat{u}$ exist on $B \setminus E$ and is quasicontinuous.
Continuity of $u_i$ implies that
\begin{align*}
\limsup_{r \to 0} & \left| u_i(x)-\fint_{B(x,r)} u(y) \, dy \right| \\
& \leq \limsup_{r \to 0} \left( \fint_{B(x,r)} |u_i(x)-u_i(y)| \, dy + \fint_{B(x,r)} |u_i(y)-u(y)| \, dy \right) \\
& \leq \limsup_{r \to 0} \fint_{B(x,r)} |u_i(y)-u(y)| \, dy \leq M(u_i-u)(x).
\end{align*}
If $x \in B \setminus V_k$, then for any $i,j \geq k$ we have
\begin{align*}
|u_i(x) - u_j(x)| &
\leq  \limsup_{r \to 0} \left( \left| u_i(x)-\fint_{B(x,r)} u(y) \, dy \right| + \left| u_j(x)-\fint_{B(x,r)} u(y) \, dy \right| \right) \\
& \leq M(u_i-u)(x) + M(u_j-u)(x) \leq 2^{-i} + 2^{-j}.
\end{align*}
It follows that the pointwise limit function $v(x) := \lim_{i\to\infty} u_i(x)$ exists for $x \in B \setminus V_k$ for every $k$, hence $v$ exists on $B \setminus E$. Since the convergence is uniform on $B \setminus V_k$, it follows that $v|_{B \setminus V_k}$ is continuous, which shows that $v$ is quasicontinuous. Then we show that $v = \hat{u}$ on $B \setminus E$. Fix a point $x$ in $B \setminus E$. Then
\[
\limsup_{r \to 0} \left| v(x) - \fint_{B(x,r)} u(y) \, dy \right| \leq |v(x)-u_i(x)| + \limsup_{r \to 0} \left|u_i(x)-\fint_{B(x,r)} u(y) \, dy \right|.
\]
Since the right-hand side approaches $0$ as $i \to \infty$, and the left-hand side does not depend on $i$, it follows that the left-hand side equals $0$, and thus $v(x) = \hat{u}(x)$. To finish the proof, we note that almost every point is Lebesgue point of $u$, and it follows that $u = \hat{u}$ almost everywhere.
\end{proof} 

In the following corollary, we show that assumption $u \in W^{1,\phix}(\Rn)$  can be replaced by $u \in W^{1,\phix}(\Omega)$. 

\begin{cor}\label{cor:quasicont_represent}
Let $\phi \in \Phi_c(\Rn)$ satisfy (A0), (A1), (aInc) and (aDec).
Let $u \in W^{1, \phix}(\Omega)$.
Then there exists  a set $E \subset \Omega$ of zero capacity, such that $\hat{u}(x)$ exists for every $x \in \Omega \setminus E$.
Moreover $\hat{u}$ is quasicontinuous in $\Omega$.
\end{cor}

\begin{proof}
Let $B$ be a ball such that $\Omega \subset B$.
Let $U$ and $V$ be open sets such that $U \Subset V \Subset \Omega$.
Let $\eta \in C_0^\infty(\Omega)$ be such that $\eta = 1$ on $V$ and $0 \leq \eta \leq 1$ on $\Omega$.
Then $u\eta \in W^{1, \phix}(\Omega)$.
Since $\spt u\eta \subset \Omega$, by \cite[Lemma~3.4]{HarHT17}, $u\eta \in W_0^{1, \phix}(\Omega) \subset W_0^{1, \phix}(\Rn)$.
Lemma \ref{lem:quasicont_represent} shows that there is a set $E \subset \Rn$ of zero capacity such that the limit
\[
\lim_{r \to 0} \fint_{B(x,r)} u(y)\eta(y) \, dy
\]
exist everywhere on $B \setminus E$. Since $U \Subset V$, $V$ is open, and $\eta = 1$ on $V$, we have
\[
\hat{u}(x) =
\lim_{r \to 0} \fint_{B(x,r)} u(y) \, dy =
\lim_{r \to 0} \fint_{B(x,r)} u(y)\eta(y) \, dy
\]
for every $x \in U \setminus E$.

Let then $(U_i)$ be a sequence of open sets such that $U_i \Subset U_{i+1} \Subset \Omega$ and $\bigcup_{i=1}^{\infty} U_i = \Omega$.
Then for every $i$ there exist a set $E_i$ of zero capacity, such that $\hat{u}$ exist in $U_i \setminus E_i$.
It follows that $\hat{u}$ exists in $\Omega \setminus \bigcup_{i=1}^{\infty} E_i$.
By subadditivity, $C_{\phix}(\bigcup_{i=1}^{\infty} E_i) = 0$.

It remains to show quasicontinuity.
By Lemma \ref{lem:quasicont_represent} $\hat{u}$ is quasicontinuous on every $U_i$.
Hence we may choose open sets $F_i$ such that $C_{\phix}(F_i) < 2^{-i}\varepsilon$ and $\hat{u}|_{U_i \setminus F_i}$ is continuous.
Hence $\bigcup_{i=1}^{\infty} F_i$ is open, $\hat{u}|_{\Omega \setminus \bigcup_{i=1}^{\infty} F_i}$ is continuous, and by subadditivity $C_{\phix}(\bigcup_{i=1}^{\infty} F_i) < \varepsilon$.
\end{proof}

Now we can prove that $u = u^*$ in Lemma \ref{lem:lsc}.

\begin{lem}\label{lem:equals_qe}
Let $\phi \in \Phi_c(\Rn)$ satisfy (A0), (A1), (A1-$n$), (aInc) and (aDec).
Let $u$ and $u^*$ be as in Lemma \ref{lem:lsc}.
If $u$ is quasicontinuous, then $u = u^*$ quasieverywhere.
\end{lem}

\begin{proof}
Suppose that $u$ is quasicontinuous. For any positive integer $k$, we let $u_k = \min \{u, k\}$. It is easy to see that $u_k$ is quasicontinuous.
By Corollary \ref{cor:quasicont_represent} there exists a set $E_k$ of zero capcacity such that
\[
\hat{u}_k(x) :=
\lim_{r \to 0} \fint_{B(x,r)} u_k(y) \, dy
\]
exist for all $x \in \Omega \setminus E_k$, and $\hat{u}_k$ is quasicontinuous.
Since both $u_k$ and $\hat{u}_k$ are quasicontinous and $u_k = \hat{u}_k$ almost everywhere in $\Omega$, it follows by \cite{Kil98} that $u_k = \hat{u}_k$ quasieverywhere in $\Omega$.
Let $F_k = E_k \cup \{u_k \neq \hat{u}_k\}$.
Then $C_{\phix}(F_k)=0$ and we have
\[
u_k(x) =
\lim_{r \to 0} \fint_{B(x,r)} u_k(y) \, dy
\]
for every $x \in \Omega \setminus F_k$.

By \cite[Lemma~4.6]{HarH19}, $u_k$ is a quasisuperminimizer.
By our assumption, $u$ is bounded from below, hence $u_k$ is bounded. 
Let $u_k^*$ be defined by
\[
u_k^*(x) :=
\essliminf_{y \to x} u_k(y). 
\]
By Lemma \ref{lem:lsc}, every point of $\Omega$ is a Lebesgue point $u_k^*$, and $u_k = u_k^*$ almost everywhere.
Hence, for every $x \in \Omega \setminus F_k$, we have
\[
u_k(x) =
\lim_{r \to 0} \fint_{B(x,r)} u_k(y) \, dy =
\lim_{r \to 0} \fint_{B(x,r)} u_k^*(y) \, dy =
u_k^*(x).
\]
Therefore $\{u_k \neq u_k^*\} \subset F_k$, and it follows that $C_{\phix}(\{u_k \neq u_k^*\}) = 0$.

Let
\[
A_k^1 :=
\{x \in \Omega : u^*(x) < u(x) \leq k\}
\text{ and }
A_k^2 :=
\{x \in \Omega :  u(x) < u^*(x) \text{ and } u(x) \leq k\}.
\]
We will show that $A_k^1 \subset \{u_k \neq u_k^*\}$ and  $A_k^2 \subset \{u_{2k} \neq u_{2k}^*\}$, which shows that both $A_k^1$ and $A_k^2$ are of capacity zero.
Let $x_0 \in \Omega$ be such that $u(x_0) \leq k$.
Suppose first that $u^*(x_0) < u(x_0)$.
Let $r > 0$ be so small that $B(x_0,r) \subset \Omega$.
Then $\essinf_{y \in B(x_0,r)} u(y) \leq u^*(x_0) < k$,
from which it follows that
\[
\essinf_{y \in B(x_0,r)} \min \{u(y), k\} =
\min \{\essinf_{y \in B(x_0,r)} u(y), k\} \leq u^*(x_0).
\]
Hence 
\[
u_k^*(x_0) =
\essliminf_{y \to x_0} u_k(y) \leq
u^*(x_0) < u(x_0) =
u_k(x_0),
\]
and $x_0 \in \{u_k \neq u_k^*\}$.
Suppose then that $u^*(x_0) > u(x_0)$.
Then there exists $r_0 > 0$ such that $B(x_0,r_0) \subset \Omega$ and $\essinf_{y \in B(x_0,r_0)} u(y) > u(x_0)$.
Hence
\[
u_{2k}^*(x_0) \geq
\essinf_{y \in B(x_0,r_0)} \min \{u(y), 2k\} =
\min \{\essinf_{y \in B(x_0,r_0)} u(y), 2k\} > 
u(x_0) =
u_{2k}(x_0),
\]
and $x_0 \in \{u_{2k} \neq u_{2k}^*\}$.

Since $A := \{x \in \Omega : u(x) \neq u^*(x) \text{ and } u(x) < \infty\} = \bigcup_{k = 1}^\infty (A_k^1 \cup A_k^2)$, we get by subadditivity that $C_{\phix}(A) = 0$.
Since $u$ is quasicontinuous $C_{\phix}(\{u = \infty\}) = 0$, and therefore $A' := \{x \in \Omega : u(x) \neq u^*(x) \text{ and } u(x) = \infty\}$ is of capacity zero.
And finally, since $\{u \neq u^*\} = A \cup A'$, we get $C_{\phix}(\{u \neq u^*\}) = 0$.
\end{proof}

\section{The Kellogg property}

In this section we prove our main result.
But first, we have to we prove some auxiliary results.
The next lemma gives a characterization of $W^{1, \phix}(\Omega)$ using quasicontinuous functions (cf. \cite[Proposition~2.5]{AdaBB14}).

\begin{lem}\label{lem:zero_bound_char}
Let $\phi \in \Phi_c(\Rn)$ satisfy (A0), (A1), (aInc) and (aDec).
Assume that $u$ is quasicontinuous in $\Omega$. Then $u \in W_0^{1, \phix}(\Omega)$ if and only if
\[
\tilde{u} :=
\left\{ \begin{array}{ll}
	u & \text{in } \Omega, \\
	0 & \text{otherwise},
\end{array} \right.
\]
is quasicontinous and belongs to $W^{1, \phix}(\Rn)$.
\end{lem}

\begin{proof}
Suppose first that $u \in W_0^{1, \phix}(\Omega)$.
By definition of $W_0^{1, \phix}(\Omega)$, there are functions $v_i \in C_0^{\infty}(\Omega)$ such that $v_i \to u$ in $W^{1, \phix}(\Omega)$.
Then $v_i \to u$ in $W^{1, \phix}(\Rn)$ also.
By \cite[Lemma~5.1]{BarHH18}, we may assume that $v_i$ converges pointwise quasieverywhere, and that the convergence is uniform outside a set of arbitrarily small capacity.
Denote the pointwise limit of $\{v_i\}$ by $v$. Then $v$ is quasicontinuous and $v = 0$ quasieverywhere in $\Rn \setminus \Omega$.
Since $u = v$ almost everywhere in $\Omega$, and both functions are quasicontinuous in $\Omega$, it follows from \cite{Kil98} that $u = v$ quasieverywhere in $\Omega$.
It then follows that $\tilde{u} = v$ quasieverywhere in $\Rn$, hence $\tilde{u}$ is quasicontinuous and belongs to $W^{1, \phix}(\Rn)$.

Suppose then that $\tilde{u}$ is quasicontinuous and belongs to $W^{1, \phix}(\Rn)$.
Let $B$ be an open ball such that $\Omega \Subset B$.
Then $\tilde{u} \in W^{1, \phix}(B)$.
Since $\tilde{u} = 0$ in $B \setminus \Omega$, it follows by Lemma \ref{lem:nolla_ulkopuolella} that $\tilde{u} \in W_0^{1, \phix}(\Omega)$, hence $u \in W_0^{1, \phix}(\Omega)$ also.
\end{proof}

Then we need the comparison principle given by the corollary following the next lemma.

\begin{lem}[Proposition~4.9 in \cite{Kar_pp}]\label{lem:comparison_principle}
Let $\phi\in \Phi_c (\Omega)$ be strictly convex and satisfy (A0), (A1) and (aDec). If $f, g \in W^{1, \phi}(\Omega)$ and $(f -g)_+ \in W^{1, \phi}_0(\Omega)$, then $H(f) \le H(g)$ in $\Omega$. 
\end{lem}

\begin{cor}\label{cor:comparison_cont_bound}
Let $\phi\in \Phi_c (\Omega)$ be strictly convex and satisfy (A0), (A1), (aInc) and (aDec).
If $f, g \in C(\partial\Omega)$ and $f \leq g$ quasieverywhere on $\partial\Omega$, then $H_f \le H_g$ in $\Omega$.
\end{cor}

\begin{proof}
Suppose first that $f, g \in \lip(\partial\Omega)$.
Extend them to Lipschitz functions defined on the whole $\Rn$.
If we show that $(f-g)_+ \in W_0^{1,\phix}(\Omega)$, then the claim follows from Lemma \ref{lem:comparison_principle}.
Let $\ve > 0$.
Since $C_\phix(\{f > g\} \cap \partial\Omega) = 0$, there is an open set $U \supset \{f > g\} \cap \partial\Omega$ and a function $u_\ve \in W^{1,\phix}(\Rn)$, such that $u_\ve = 1$ in $U$, $0 \leq u_\ve \leq 1$, and
\[
\int_{\Rn} \phi(x,|u_\ve|) + \phi(x,|\nabla u_\ve|) \, dx < \ve. 
\]
Since $f$ and $g$ are continuous, the set $V :=\{f < g+\ve\}$ is open.
It is true that $\partial\Omega \subset U \cup V$.
Let $v_\ve := \chi_\Omega(1-u_\ve)((f-g)_+-\ve)_+$.
Then $\spt v_\ve \subset \Omega \setminus (U \cup V) \Subset \Omega$, hence we may choose an open set $D$, such that $\spt v_\ve \Subset D \Subset \Omega$.
It follows from Lemma \ref{lem:nolla_ulkopuolella} that $v_\ve \in W_0^{1,\phix}(\Omega)$.
A straightforward calculation shows that $v_\ve \to (f-g)_+$ in $W^{1,\phix}(\Omega)$, hence $(f-g)_+ \in W_0^{1,\phix}(\Omega)$.

Suppose then $f, g \in C(\partial\Omega)$.
Fix $x \in \Omega$ and let $\ve > 0$.
Let $\eta \in \lip(\partial\Omega)$ be such that $\eta \leq f$ on $\partial\Omega$ and $H(\eta)(x) > H_f(x) - \ve$.
Let $\xi \in \lip(\partial\Omega)$ be such that $g \geq \xi \geq g-\ve$ on $\partial\Omega$.
If $y \in \partial\Omega$ is such that $f(y) \leq g(y)$, then $\eta(y) - \ve \leq \xi(y)$.
Hence $\eta-\ve \leq \xi$ quasieverywhere in $\partial\Omega$.
By the first part of the proof, it follows that $H(\eta-\ve) \leq H(\xi)$ on $\Omega$. Now
\[
H_f(x)-\ve < H(\eta)(x) = H(\eta-\ve)(x)+\ve \leq H(\xi)(x)+\ve \leq H_g(x)+\ve. 
\]
Since $\ve$ was arbitrary $H_f(x) \leq H_g(x)$, and the claim now follows, since $x$ was arbitrary.
\end{proof}

We need one more lemma in order to prove our main result.
This lemma corresponds to \cite[Lemma~5.5]{AdaBB14}, and the proof is also similar, but we include it here for completeness.

\begin{lem}\label{lem:B_supersol}
Suppose $\phi \in \Phi_c(\Rn)$ is strictly convex and satisfies (A0), (A1), (A1-$n$), (aInc) and (aDec). Let $x \in \partial\Omega$ and $B := B(x, r)$.
Let $f$ be Lipschitz on $\partial\Omega$ and suppose that $f = m$ on $B \cap \partial\Omega$, where $m := \sup_{\partial\Omega} f$. Let (see Figure \ref{fig:kuva})
\[
u := \left\{\begin{array}{ll}
	H(f) & \text{in } \Omega, \\
	m & \text{in } B \setminus \Omega.
\end{array}\right.
\]
Then $u$ is a quasicontinuous superminimizer in $B$.
\end{lem}

\begin{figure}[ht!]
\begin{tikzpicture}[line cap=round,line join=round,>=triangle 45,x=1.0cm,y=1.0cm]
\clip(-6.5,-3) rectangle (7.32,3.2);
\draw [shift={(3.,3.)},line width=1.pt]  plot[domain=3.9269908169872414:5.497787143782138,variable=\t]({1.*4.242640687119286*cos(\t r)+0.*4.242640687119286*sin(\t r)},{0.*4.242640687119286*cos(\t r)+1.*4.242640687119286*sin(\t r)});
\draw [shift={(-4.,-3.)},line width=1.pt]  plot[domain=0.6435011087932844:2,variable=\t]({1.*5.*cos(\t r)+0.*5.*sin(\t r)},{0.*5.*cos(\t r)+1.*5.*sin(\t r)});
\draw [line width=1.pt] (0.,0.) circle (3.cm);
\draw [line width=1.pt] (-6.2,-1.)-- (-4.,-1.);
\draw [line width=1.pt] (-4.,-1.)-- (-4.,-2.);
\draw [line width=1.pt] (-4.,-2.)-- (-6.2,-2.);
\draw [line width=1.pt] (-6.2,-2.)-- (-6.2,-1.);
\draw [line width=1.pt] (4.,1.)-- (4.,2.);
\draw [line width=1.pt] (4.,2.)-- (5.7,2.);
\draw [line width=1.pt] (5.7,2.)-- (5.7,1.);
\draw [line width=1.pt] (5.7,1.)-- (4.,1.);
\draw (-6.1,-1.2) node[anchor=north west] {$u = H(f)$};
\draw (4.1,1.7) node[anchor=north west] {$u = m$};
\draw [->,line width=1.pt] (-4.,-1.6) -- (-1.8,-1.);
\draw [->,line width=1.pt] (4.,1.4) -- (1.6,0.82);
\draw (0.16,0.4) node[anchor=north west] {$x$};
\draw (-5.98,2.4) node[anchor=north west] {$\partial\Omega$};
\draw (0.54,2.8) node[anchor=north west] {$B$};
\draw (-5.34,0.6) node[anchor=north west] {$\Omega$};
\begin{scriptsize}
\draw [fill=black] (0.,0.) circle (2.0pt);
\end{scriptsize}
\end{tikzpicture}
\caption{Definition of $u$ in Lemma \ref{lem:B_supersol}}\label{fig:kuva}
\end{figure}
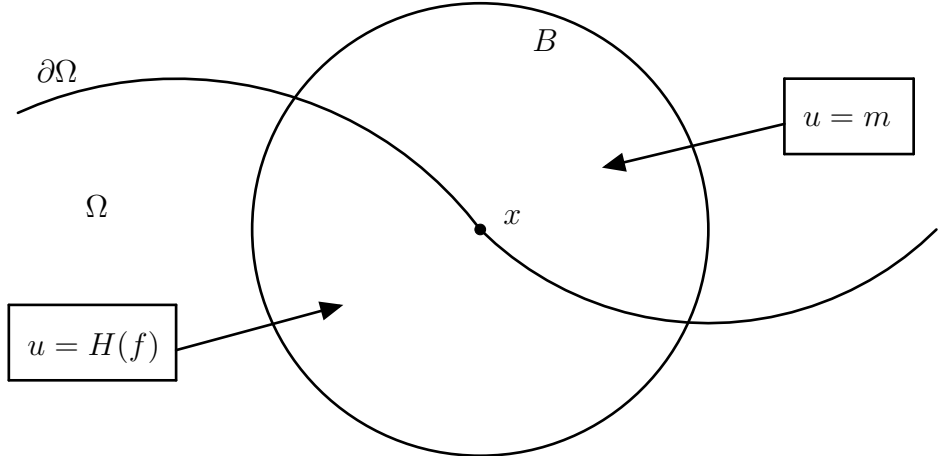

\begin{proof}
Extend $f$ to a Lipschitz function defined on $\overline{\Omega}$ in such a way that $f \leq m$.
Extend $f$ further by setting $f := m$ on $\overline{B} \setminus \Omega$.
Then $f \in W^{1, \phix}(B)$.
Let
\[
v := \left\{\begin{array}{ll}
	u - f & \text{in }B \cup \Omega, \\
	0 & \text{otherwise.}
\end{array}\right.
\]
As $v = H(f) - f$ in $\Omega$, we have $v \in W_0^{1, \phix}(\Omega)$.
It follows from Theorem \ref{thm:H(f)}, that $v$ is continuous in $\Omega$.
Since $v = 0$ in $\Rn \setminus \Omega$, it follows from Lemma \ref{lem:zero_bound_char} that $v$ is quasicontinuous  in $B$ and and belongs to $W^{1, \phix}(B)$.
It now follows that $u$ is quasicontinuous and belongs to $W^{1, \phix}(B)$.
By Corollary \ref{cor:comparison_cont_bound}, $u \leq m$ in $B$.

Now we show that $u$ is a superminimizer.
Let $\eta \in C_0^\infty(B)$ be nonnegative and let $\eta' := \min\{\eta, m - u\}$.
It is easy to see that $\eta'$ is quasicontinuous and nonnegative in $B$.
By \cite[Lemma~2.11]{HarH19}, $\eta' \in W_0^{1, \phix}(B)$.
Since $\eta' = 0$ in  $B \setminus \Omega$, it follows from Lemma \ref{lem:zero_bound_char} that $\eta' \in W_0^{1, \phix}(B \cap \Omega)$.
Now we have
\[
\int_{\{\eta \neq 0\}}\phi(x, |\nabla u|) \, dx
= \int_{\{\eta' \neq 0\}}\phi(x, |\nabla u|) \, dx
\leq \int_{\{\eta' \neq 0\}}\phi(x, |\nabla (u + \eta')|) \, dx.
\]
The equality above follows from the facts that $\{\eta' = 0 \neq \eta\} \subset \{u = m\}$ and $\nabla u = 0$ almost everywhere in $\{u = m\}$.
The inequality follows from the facts that $\{ \eta' \neq 0 \} \subset \Omega$ and $u$ is a minimizer in $\Omega$.
Since $u + \eta' = \min \{u + \eta, m\}$, we have $|\nabla (u + \eta')| \leq |\nabla (u + \eta)|$.
And since $\eta' \neq 0$ implies $\eta \neq 0$, we get
\[
\int_{\{\eta' \neq 0\}}\phi(x, |\nabla (u + \eta')|) \, dx
\leq \int_{\{\eta \neq 0\}}\phi(x, |\nabla (u + \eta)|) \, dx.
\]
Combining the estimates above and using Lemma \ref{lem:min_cont}, we see that $u$ is a superminimizer in $B$.
\end{proof}

We are now ready to prove our main result.
The proof is again similar to the proof of \cite[Theorem~1.1]{AdaBB14}, but is included here for completeness.

\begin{thm}\label{thm:main}
Let $\phi \in \Phi_c(\Rn)$ be strictly convex and satisfy (A0), (A1), (A1-n), (aDec) and (aInc).
Then the set of irregular boundary points has zero capacity.
\end{thm}

\begin{proof}
Denote the set of irregular points by $I$. To prove that $I$ is of capacity zero, we construct a countable number of sets $I_{j,k,q} \subset I$, such that $C_{\phix}(I_{j,k,q}) = 0$, and the union of sets $I_{j,k,q}$ is equal to $I$. 

For any positive integer $j$ we can cover $\partial\Omega$ with a finitely many balls $B_{j,k} := B(x_{j,k}, 1/j)$, $1 \leq k \leq N_j$.
Let $v_{j,k}$ be a Lipschitz function such that $\supp v_{j,k} \subset 3B_{j,k}$, $0 \leq v_{j,k} \leq 1$, and $v_{j,k} = 1$ on $2B_{j,k}$.
For any positive $q \in \mathbb{Q}$, let $v_{j,k,q} = qv_{j,k}$.
Consider the sets
\[
I_{j,k,q} :=
\{ x \in \overline{B}_{j,k} \cap \partial\Omega : \liminf_{\Omega \ni y \to x} H(v_{j,k,q})(y) < v_{j,k,q}(x) = q \}.
\]
Then $I_{j,k,q} \subset I$. To show that $I_{j,k,q}$ is of capacity zero, let
\[
u_{j,k,q} := \left\{ \begin{array}{ll}
	H(v_{j,k,q}) & \text{in }\Omega, \\
	q & \text{in }2B_{j,k} \setminus \Omega.
\end{array}\right.
\]
By Lemma \ref{lem:B_supersol}, $u_{j,k,q}$ is a quasicontinuous superminimizer in $2B_{j,k}$, and by Corollary \ref{cor:comparison_cont_bound}, $u_{j,k,q} \leq q$ in $\Omega$.
Since $u_{j,k,q}$ is continuous in $\Omega$, for every  $x \in \Omega$ we have
\[
u_{j,k,q}^*(x) :=
\essliminf_{y \to x} u_{j,k,q}(y) =
u_{j,k,q}(x) =
H(v_{j,k,q})(x).
\]
By Lemma \ref{lem:lsc}, $u_{j,k,q}^*$ is lower semicontinuous, and, by Lemma \ref{lem:equals_qe}, $u_{j,k,q}^* = u_{j,k,q}$ quasieverywhere in $2B_{j,k}$.
Since $u_{j,k,q}^* \leq q$, we have  
\[
q =
u_{j,k,q}(x) =
u_{j,k,q}^*(x) =
\liminf_{\Omega \ni y \to x} u_{j,k,q}^*(y) =
\liminf_{\Omega \ni y \to x} H(v_{j,k,q})(y)
\]
for quasievery $x \in \overline{B}_{j,k} \cap \partial\Omega$.
Hence $I_{j,k,q}$ is of capacity zero.

Then we show that every point of $I$ belongs to some $I_{j,k,q}$.
Let therefore $x \in I$.
Then there exists a function $v \in C(\partial\Omega)$ such that
\[
\lim_{\Omega \ni y \to x} H_v(y) \neq
v(x).
\]
By considering $-v$ if necessary, we may assume that $\liminf_{\Omega \ni y \to x} H_v(y) < v(x)$, and by adding a constant, we may assume that $v \geq 0$. Since $v$ is continuous, we can find a ball $B_{j,k} \ni x$ such that
\[
m :=
\inf_{3B_{j,k} \cap \partial\Omega} v >
\liminf_{\Omega \ni y \to x} H_v(y).
\]
We can then choose $q \in \mathbb{Q}$ such that $m > q > \liminf_{\Omega \ni y \to x} H_v(y)$.
Then $v_{j,k,q} \leq v$ on $\partial\Omega$, and it follows by Corollary \ref{cor:comparison_cont_bound}, that
\[
\liminf_{\Omega \ni y \to x} H(v_{j,k,q})(y) \leq
\liminf_{\Omega \ni y \to x} H_v(y) <
q =
v_{j,k,q}(x).
\]
But then $x \in I_{j,k,q}$.

We have now shown that
\[
 I =
 \bigcup_{j=1}^\infty \bigcup_{k=1}^{N_j} \bigcup_{\substack{q\in\mathbb{Q} \\ q>0}} I_{j,k,q}.
\]
It now follows by subadditivity that $I$ is of zero capacity.
\end{proof}


\section{Semiregular boundary points}

In this section we give some characterizations of semiregular boundary points.
We follow the ideas in \cite[Section~8]{AdaBB14}, where characterizations of semiregular boundary points are given in the variable exponent case.

\begin{defn}\label{defn:semiregular}
A point $x \in \partial\Omega$ is semiregular, if it is irregular, and the limit
\[
\lim_{\Omega \ni y \to x} H_f(y)
\]
exists for every $f \in C(\partial\Omega)$.
\end{defn}

First we prove some lemmas.

\begin{lem}\label{lem:jono}
Let $\phi \in \Phi_c(\Rn)$ satisfy (A0), (A1), (aInc) and (aDec), and let $K \subset \Rn$ be compact with $C_\phix(K)=0$. 
Then there exists a sequence of functions $\xi_i \in C^\infty(\Rn)$, with the following properties:
\begin{itemize}
\item[(i)] $0\leq \xi_i\leq 1$ in $\Rn$ and $\xi_i = 0$ in a neighbourhood of $K$,
\item[(ii)] $\lim_{i \to \infty} \|1-\xi_i\|_\phix=0=\lim_{i\to\infty}\|\nabla\xi_i\|_\phix$,
\item[(iii)] $\lim_{i \to \infty}\xi_i(x)=1$ and $\lim_{i \to \infty}\nabla\xi_i(x)=0$ for almost every $x\in\Rn$.
\end{itemize}
\end{lem}

\begin{proof}
Let $i$ be a positive integer and let $u$ be a test function for capacity of $K$ with
\[
\int_{\Rn}\phi(x,|u|)+\phi(x,|\nabla u|)\,dx < \frac{1}{i}.
\]
Since $\min\{u,1\}$ is also a test function, we may assume that $0 \leq u \leq 1$.
Let $U$ be an open set containing $K$, such that $u = 1$ in $U$.
Let $\eta \in C_0^\infty(\Rn)$ be such that $\eta = 1$ on $U$, $0\leq\eta\leq 1$ in $\Rn$ and $|\nabla\eta| \leq 1$.
Then $u\eta \in W^{1,\phix}(\Rn)$, and using triangle inequality we get
\[
\| u\eta \|_{1,\phix} \leq \| u\eta \|_{\phix} + \| u\nabla\eta \|_{\phix} + \| \eta\nabla u \|_{\phix} \leq 2\| u \|_{\phix} + \| \nabla u \|_{\phix}.
\]
By \cite[Lemma~3.2.11]{HarH19b} there is a constant $c$ such that
\[
\| u \|_{\phix} \leq c \max\{\varrho_\phix(u),\varrho_\phix(u)^\frac{1}{q}\},
\]
where $q$ is the exponent from (aDec). Since $\varrho_\phix(u) < 1/i \leq 1$, the maximum above equals $\varrho_\phix(u)^\frac{1}{q}$, hence $\| u \|_{\phix} \leq c(1/i)^\frac{1}{q}$.
Similarly we get $\| \nabla u \|_{\phix} \leq c(1/i)^\frac{1}{q}$.
Combining all the estimates gives $\| u\eta \|_{1,\phix} < 3c(1/i)^\frac{1}{q}=:\ve_i$.

Let now $B$ be an open ball such that $\spt\eta \Subset B$.
By \cite[Theorem~6.4.7]{HarH19b} and \cite[Lemma~4.2.3]{HarH19b} there exists a sequence of functions $\mu_j \in C^\infty(B) \cap W^{1,\phix}(B)$ converging to $u$ in $W^{1,\phix}(B)$.
Since $u=1$ on $U$, and the proof of \cite[Theorem~6.4.7]{HarH19b} uses standard mollifiers, we have $0 \leq \mu_j \leq 1$ on $B$.
Moreover, we may assume that $\mu_j=1$ on an open set $V$, with $K\subset V\subset U$.
Now $\mu_j\eta \in C_0^\infty(\Rn)$ and $\mu_j\eta \to u\eta$ in $W^{1,\phix}(\Rn)$.
Let $j_i$ be an index such that $\|\mu_{j_i}\eta\|_{1,\phix}<2\ve_i$, and let $\nu_i:=\mu_{j_i}\eta$.

Now $\|\nu_i\|_{\phix}\leq\|\nu_i\|_{1,\phix}<2\ve_i\to 0$ as $i\to\infty$.
Similarly $\|\nabla\nu_i\|_\phix \to 0$.
It now follows from \cite[Lemma~3.3.6]{HarH19b} that we may choose a subsequence $\nu_{i_k}$ such that $\nu_{i_k}$ and $\nabla\nu_{i_k}$ converge to $0$ pointwise almost everywhere.
Choosing $\xi_k := 1-\nu_{i_k}$ we get a sequence satisfying properties (i), (ii) and (iii).
\end{proof}

Next we prove a lemma concerning extension of lsc-regularized superminimizers.

\begin{lem}\label{lem:extension}
Let $\phi \in \Phi_c(\Rn)$ be strictly convex and satisfy (A0), (A1), (aInc) and (aDec).
Let $F \subset \Omega$ be relatively closed with $C_\phix(F) = 0$, and let $u \in W^{1,\phix}(\Omega\setminus F)$ be a bounded lsc-regularized superminimizer in $\Omega \setminus F$.
Then $u$ has a unique bounded lsc-regularized extension $v \in W^{1,\phix}(\Omega)$, given by
\[
v(x) := \essliminf_{\Omega\setminus F \ni y \to x} u(x).
\]
Moreover, $v$ is a superminimizer in $\Omega$.
\end{lem}

\begin{proof}
By \cite[Lemma~4.1]{BarHH18}, we have $|F| = 0$. Existence of $v$ is therefore trivial, and boundedness of $v$ follows easily from boundedness of $u$.
Since
\[
\essliminf_{\Omega \ni y \to x} v(x) = \essliminf_{\Omega\setminus F \ni y \to x} u(x) = v(x)
\]
for all $x\in\Omega$, and $v$ is lsc-regularized.
The equality above also implies that $v$ is unique.
That $v\in L^{\phix}(\Omega)$ follows directly from the facts that $u\in L^{\phix}(\Omega\setminus F)$ and $|F|=0$.

Now we show that $\partial_j v = \partial_j u$ for $j=1,2,\dots,n$.
Let $\eta \in C_0^\infty(\Omega)$, and let $K := F \cap \spt \eta$.
Then $K$ is compact and $C_\phix(K)=0$, and we can find a sequence $\{\xi_i\}$ as in Lemma \ref{lem:jono}.
Now $\eta\xi_i \in C_0^\infty(\Omega\setminus F)$.
The definition of weak derivative gives
\[
0
=\int_{\Omega\setminus F} u\partial_j(\eta\xi_i)+\eta\xi_i\partial_j u\,dx
=\int_\Omega v\eta\partial_j\xi_i+\xi_i(v\partial_j\eta+\eta\partial_j u)\,dx,
\]
where we have also used the fact that $v=u$ almost everywhere in $\Omega$.
Hölder's inequality gives
\[
\left|\int_\Omega v\eta\partial_j\xi_i\,dx\right| \leq \int_\Omega |v\eta\partial_j\xi_i|\,dx \leq 2\|\partial_j\xi_i\|_\phix\|v\eta\|_{\phi^*(\cdot)}.
\]
By \cite[Proposition~2.4.13]{HarH19b} and \cite[Lemma~3.7.6]{HarH19b} $\phi^*$ satisfies (aDec) and (A0), which implies that $\phi^*(x,1)\lesssim 1$.
Since $v\eta$ is bounded, it follows that $v\eta \in L^{\phi^*(\cdot)}(\Omega)$.
It now follows from property (ii) in Lemma \ref{lem:jono}, that the right-hand side in the inequality above approaches $0$ as $i\to\infty$.
By \cite[Corollary~3.7.9]{HarH19b}, $L^{\phix}(\Omega)\subset L^p(\Omega)$, where $p$ is the exponent from (aInc).
Since $\eta \in C^\infty_0(\Omega)$ and $v=u \in W^{1, \phi(\cdot)}(\Omega)$, it follows that $(v\partial_j\eta+\eta\partial_j u) \in L^1(\Omega)$.
Since $\xi_i\leq 1$ and $\xi_i(x)\to 1$ for almost every $x\in\Omega$, dominated convergence implies that
\[
\lim_{i\to\infty}\int_\Omega \xi_i(v\partial_j\eta+\eta\partial_j u)\,dx=\int_\Omega (v\partial_j\eta+\eta\partial_j u)\,dx.
\]
Combining the results above shows that $\partial_j v=\partial_j u$.

To complete the proof, we need to show that $v$ is a superminimizer in $\Omega$.
Let $0 \leq \mu \in C_0^\infty$ and let $\xi_i$ be as above.
Denote $w_i:=v+\eta\xi_i$ and $w:=v+\eta$.
Since $u$ is a superminimizer in $\Omega\setminus F$ and $\eta\xi_i\in C_0^\infty(\Omega\setminus F)$, we have
\[
\int_\Omega \phi(x,|\nabla w_i|)\,dx
=\int_{\Omega\setminus F} \phi(x,|\nabla (u+\eta\xi_i)|)\,dx
\geq \int_{\Omega\setminus F} \phi(x,|\nabla u|)\,dx
=\int_{\Omega} \phi(x,|\nabla v|)\,dx.
\]
The claim follows, if we can show that $\lim_{i\to\infty}\varrho_\phix(\nabla w_i)=\varrho_\phix(\nabla w)$.
Since
\[
\big{|}|\nabla w_i|-|\nabla w|\big{|} \leq |\nabla(w_i-w)| = |\nabla(\eta\xi_i-\eta)|,
\]
we have
\begin{align*}
& \varrho_\phix(|\nabla w_i|-|\nabla w|) \leq \varrho_\phix(\nabla (\eta\xi_i-\eta))=\varrho_\phix(\nabla\eta(\xi_i-1)+\eta\nabla\xi_i) \\
& \leq \varrho_\phix(|\nabla\eta(\xi_i-1)|+|\eta\nabla\xi_i|)\lesssim \varrho_\phix(\nabla\eta(\xi_i-1))+\varrho_\phix(\eta\nabla\xi_i),
\end{align*}
where the last inequality follows from (aDec).
Property (ii) of Lemma \ref{lem:jono} implies that $\|\eta\nabla\xi_i\|_\phix$ tends to $0$ as $i\to\infty$, and \cite[Lemma~3.2.11]{HarH19b} then implies that $\lim_{i\to\infty}\varrho_\phix(\eta\nabla\xi_i)=0$.
Similarly $\lim_{i\to\infty}\varrho_\phix(\nabla\eta(\xi_i-1))=0$.
It now follows from \cite[Lemma~3.1.6]{HarH19b} that $\lim_{i\to\infty}\varrho_\phix(|\nabla w_i|)=\varrho_\phix(|\nabla w|)$, which completes the proof.
\end{proof}

Now we prove the following lemma, which is our main tool in characterizing semiregular boundary points  (cf. \cite[Theorem~8.1]{AdaBB14}).

\begin{lem}\label{lem:semireg}
Let $\phi\in \Phi_c (\Omega)$ be strictly convex and satisfy (A0), (A1), (A1-n), (aInc) and (aDec).
Let $V \subset \partial\Omega$ be relatively open.
Then the following are equivalent:
\begin{itemize}
\item[(a)] Every point of $V$ is semiregular.
\item[(b)] Every point of $V$ is irregular.
\item[(c)] $C_\phix(V) = 0$.
\item[(d)] If $f,g \in C(\partial\Omega)$ and $f = g$ on $\partial\Omega \setminus V$, then $H_f = H_g$.
\end{itemize}
\end{lem}

\begin{proof}
(a) $\Rightarrow$ (b)
Follows directly from definition of semiregularity.

(b) $\Rightarrow$ (c)
Follows from the Kellogg property (Theorem \ref{thm:main}).

(c) $\Rightarrow$ (d)
Suppose that $f,g \in C(\partial\Omega)$ and $f = g$ on $\partial\Omega \setminus V$.
Since $C_\phix(V) = 0$, it follows that both $f \leq g$ and $g \leq f$ hold quasieverywhere on $\partial\Omega$.
It then follows from Corollary \ref{cor:comparison_cont_bound} that both $H_f \leq H_g$ and $H_g \leq H_f$ on $\Omega$, hence $H_f = H_g$.

(d) $\Rightarrow$ (b)
Let $x_0 \in V$.
Since $V$ is relatively open, there exists $r>0$ with $B(x_0,r) \cap \partial\Omega \subset V$.
Define $f\in\lip(\Rn)$ by $f(y) := (1-\frac{d(x_0,y)}{r})_+$.
Then $f = 0$ on $\partial\Omega \setminus V$, and it follows by our assumption that $H(f) = H(0) = 0$. Since $f(x) = 1$, it follows that $x$ is irregular.

(c) $\Rightarrow$ (a)
Let $x_0\in V$ and let $G$ be an open neighbourhood of $x_0$ such that $G\cap\partial\Omega\subset V$.
By \cite[Proposition~4.2]{BarHH18} $C_p(G\cap\partial\Omega)=0$, where $p>1$ is the exponent from (aInc).
Now \cite[Lemma~6.5]{AdaBB14} (with $p(x) = p$) implies that $G\setminus\partial\Omega$ is connected.
Since $G\setminus\partial\Omega \subset (G\cap\Omega)\cup(G\setminus\overline{\Omega})$ and $G\cap\Omega\neq\varnothing$, connectedness implies that $G\setminus\overline{\Omega}=\varnothing$.
Now
\[
G = (G\setminus\partial\Omega)\cup (G\cap\partial\Omega)\subset(G\cap\Omega)\cup V\subset \Omega\cup V.
\]
Since $x_0$ was arbitrary, this implies that $\Omega\cup V$ is open.
Moreover, since $V\subset\partial\Omega$, $V$ is relatively closed in $\Omega\cup V$.

Let $f\in\lip(\partial\Omega)$.
By Theorem \ref{thm:H(f)}, $H(f)$ is continuous, hence $H(f)$ is lsc-regularized.
By Lemma \ref{lem:extension}, $H(f)$ has a bounded lsc-regularized extension $u_f$ to $\Omega\cup V$, such that $u_f$ is a superminimizer in $\Omega\cup V$.
Lemma \ref{lem:extension} applied to $-H(f)=H(-f)$ gives an extension $u_{-f}$, and then $-u_{-f}$ is an extension of $H(f)$ that is a subminimizer $\Omega\cup V$.
By \cite[Lemma~4.1]{BarHH18}, $|V|=0$, and it follows that $u_f=-u_{-f}$ almost everywhere in $\Omega\cup V$.
Hence $u_f$ is an lsc-regularized minimizer in $\Omega\cup V$.
By \cite[Theorem~5.8]{Kar_pp} $u_f$ is continuous in an open set $\Omega \cup V$, and it follows that the limit
\[
\lim_{\Omega\ni y\to x_0} H(f)(y)=\lim_{\Omega\cup V\ni y\to x_0} u_f(y)=u_f(x_0)
\]
exists for every $x_0\in V$.

Let then $g\in C(\partial\Omega)$.
Let $f_i \in \lip(\Omega)$ be sequence such that $g-\frac{1}{i} \leq f_i \leq g$ and $f_i\leq f_{i+1}$ on $\partial\Omega$.
For any $x\in\Omega$ and $j>i$, comparison principle (Lemma \ref{lem:comparison_principle}) implies that
\[
H(f_i)(x) \leq H(f_j)(x) \leq H_g(x) \leq H(f_i)(x)+\frac{1}{i}.
\]
Hence $H(f_i)$ converges  uniformly to $H_g$ in $\Omega$, and this implies that $H_g$ is continuous in $\Omega$.
Let $u_{f_i}$ be the extension of $H(f_i)$ to $\Omega\cup V$ given by Lemma \ref{lem:extension}.
We have already shown that for any $x\in\Omega$ and $j>i$
\[
u_{f_i}(x) \leq u_{f_j}(x) \leq u_{f_i}(x)+\frac{1}{i}.
\]
Since $u_{f_i}$ are continuous in $\Omega\cup V$ and $V\subset\partial\Omega$, these inequalities hold also for every $x\in V$.
This implies that the unctions $u_{f_i}$ converge  uniformly to a continuous function $u$.
Since $u=H_g$ in $\Omega$, for any $x_0\in V$ we have
\[
\lim_{\Omega\ni y\to x_0} H_g(y)=\lim_{\Omega\cup V\ni y\to x_0} u(y)=u(x_0).
\]
The only thing left is to show that every point of $V$ is irregular.
But this follows from the already proven implication (c) $\Rightarrow$ (b).
\end{proof}

Using the previous lemma we now give some characterizations of semiregular boundary points (cf. \cite[Theorem~8.4]{AdaBB14}).

\begin{thm}\label{thm:semireg}
Let $\phi \in \Phi_c(\Rn)$ be strictly convex and satisfy (A0), (A1), (A1-n), (aDec) and (aInc). Let $x_0 \in \partial\Omega$, $\delta > 0$ and $d(y) := d(x_0,y)$. Then the following are equivalent:
\begin{itemize}
\item[(A)] The point $x_0$ is semiregular.
\item[(B)] For some positive integer $k$
\[
\lim_{\Omega \ni y \to x_0} H(kd)(y) > 0.
\]
\item[(C)] For some positive integer $k$,
\[
\liminf_{\Omega \ni y \to x_0} H(kd)(y) > 0.
\]
\item[(D)] There is no sequence $\{y_i\}$, such that $\Omega \ni y_i \to x_0$, and
\[
\lim_{i \to \infty} H_f(y_i) = f(x_0) \quad \text{for all } f \in C(\partial\Omega).
\]
\item[(E)] It is true that $x_0 \notin \overline{\{x \in \partial\Omega: x\text{ is regular}\}}$.
\item[(F)] There is a neighbourhood $V$ of $x_0$, such that $C_\phix(V \cap \partial\Omega) = 0$.
\item[(G)]  There is a neighbourhood $V$ of $x_0$ such that, for every $f,g \in C(\partial\Omega)$, if $f = g$ on $\partial\Omega \setminus V$, then $H_f = H_g$.
\item[(H)] The point $x_0$ is semiregular with respect to $G := \Omega \cap B(x_0, \delta)$.
\end{itemize}
\end{thm}

\begin{proof}
(A) $\Rightarrow$ (B)
We prove this by contraposition. Suppose therefore that
\[
\lim_{\Omega \ni y \to x_0} H(kd)(y) = 0
\]
for all positive integers $k$.
Let $f \in C_0^\infty(\Rn) \subset \lip(\partial\Omega)$, $a > f(x_0)$, and $r > 0$ be such that $f < a$  on $B(x_0,r)$.
Denote $m := \sup_{\partial\Omega}(f-a)_+$, and let $j > m/r$ be a positive integer.
Then
\[
f \leq a + \frac{md}{r} \leq a+ jd
\]
on $\partial\Omega$.
It follows from Corollary \ref{cor:comparison_cont_bound} that
\[
\limsup_{\Omega \ni y \to x_0} H(f)(y) \leq \limsup_{\Omega \ni y \to x_0} H(a + jd)(y) = \limsup _{\Omega \ni y \to x_0} (a+H(jd)(y)) = a.
\]
Letting $a \to f(x_0)$ shows that $\limsup_{\Omega \ni y \to x_0} H(f)(y) \leq f(x_0)$. Replacing $f$ with $-f$ in the calculations above gives us
\[
-\limsup_{\Omega \ni y \to x_0} H(-f)(y) \geq -(-f(x_0)) = f(x_0).
\]
Hence we have $\lim_{\Omega \ni y \to x_0} H(f)(y) = f(x_0)$. Since $f \in C_0^\infty(\Rn)$ was arbitrary, it follows from \cite[Proposition~6.5]{HarH19} that $x_0$ is regular, thus $x_0$ is not semiregular, and (A) does not hold.

(B) $\Rightarrow$ (C) $\Rightarrow$ (D)
These implications are trivial.

(D) $\Rightarrow$ (E)
We prove this by contraposition.
Suppose therefore that $x_0$ belongs to the closure of regular boundary points.
For each positive integer $k$, the intersection $B(x_0,k^{-2}) \cap \partial\Omega$ contains a regular boundary point $x_k$.
Let $f_k := kd \in \lip(\partial\Omega)$.
Then we can find $y_k \in B(x_0,k^{-1}) \cap \Omega$ with $|f_k(x_k) - H(f_k)(y_k)| < k^{-1}$. Then $y_k \to x_0$ and, since $0 \leq f(x_k) \leq k^{-1}$, we have $H(f_k)(y_k) \leq 2k^{-1}$.

Let then $f \in C(\partial\Omega)$, and assume without loss of generality that $f(x_0) = 0$.
Choose $m$ such that $|f| \leq m <\infty$ and let $\ve > 0$.
We can find $r > 0$ such that $|f| < \ve$ on $B(x_0,r^{-1}) \cap \partial\Omega$.
For $k \geq mr$ we have $f_k \geq m$ on $\partial\Omega \setminus B(x_0,r^{-1})$, hence $|f| \leq f_k + \ve$ on $\partial\Omega$.
It follows from Corollary \ref{cor:comparison_cont_bound} that for $k \geq mr$ we have
\[
H_f(y_k) \leq H(f_k)(y) + \ve \leq 2k^{-1} + \ve
\]
and
\[
H_f(y_k) \geq -H(f_k)(y) + (-\ve) \geq -2k^{-1} - \ve.
\]
Hence $-\ve \leq \lim_{k \to \infty} H_f(y_k) \leq \ve$. Since $\ve$ was arbitrary, the limit must be equal to 0, and it follows that (D) does not hold

(E) $\Leftrightarrow$ (F)
Note that (E) is equivalent to the existence of a neighbourhood $V$ of $x_0$, such that every point of $V$ is irregular. The equivalence of (E) and (F) thus follows from the equivalence (b) $\Leftrightarrow$ (c) in Lemma \ref{lem:semireg}, when we replace $V$ in \ref{lem:semireg} with $V \cap \partial\Omega$ here.

(F) $\Leftrightarrow$ (G) $\Rightarrow$ (A)
These implications follow from Lemma \ref{lem:semireg} when we replace $V$ in \ref{lem:semireg} with $V \cap \partial\Omega$ here.

(F) $\Leftrightarrow$ (H)
The assumption (F) is equivalent to the existence of a neighbourhood $W$ of $x_0$ with $C_\phix(W \cap \partial G) = 0$. This is equivalent to (H), which can bee seen by we applying the equivalence (F) $\Leftrightarrow$ (A) to the set $G$.
\end{proof}


\end{document}